\documentclass[10pt]{article}

\usepackage[a4paper,top=2.05cm,bottom=2.05cm,left=2.15cm,right=2.15cm]{geometry}
\usepackage[T1]{fontenc}
\usepackage{lmodern}
\usepackage{microtype}
\usepackage{amsmath,amssymb,amsthm,mathtools}
\usepackage{graphicx}
\usepackage{booktabs}
\usepackage{float}
\usepackage{enumitem}
\usepackage{hyperref}
\usepackage{xurl}
\usepackage{cite}
\usepackage{titlesec}
\usepackage{authblk}
\usepackage{color}
\usepackage{pgfplots}

\hypersetup{colorlinks=true,linkcolor=black,citecolor=black,urlcolor=black}
\setlist{nosep,leftmargin=1.7em}

\usepackage{placeins}

\newtheorem{theorem}{Theorem}[section]
\newtheorem*{theoremA}{Theorem A}

\newtheorem{lemma}[theorem]{Lemma}
\newtheorem{remark}[theorem]{Remark}

\title{\bfseries Quantitative finite-time approximation of\\ scattering fields for Alfv\'en waves in ideal MHD}

\author{%
	Junming Duan\thanks{
		School of Science and Engineering, The Chinese University of Hong Kong (Shenzhen), Guangdong, China.
		Email: \texttt{duanjunming@cuhk.edu.cn}.%
	}, \ \ 
	Mengni Li\thanks{(Corresponding author) School of Mathematics, Southeast University, Nanjing, China.
		Email: \texttt{krisymengni@gmail.com}.%
	}
}
\date{}

\begin{document}
\maketitle

\begin{center}
\textit{Dedicated to Professor Christian Klingenberg on the occasion of his 70th birthday}
\end{center}

\begin{abstract}
In this paper, we establish a quantitative finite-time approximation of scattering fields for nonlinear Alfv\'en waves in three-dimensional ideal incompressible magnetohydrodynamics near a strong constant magnetic background.  Using the global weighted estimates, we prove that truncating the scattering profiles at time $T$ introduces an error in both the $L^\infty$ norm and the  $H^{N_*+1}$ norm of order
$$
O\bigl(\varepsilon^2(R+T)^{-\delta}\bigr),
$$
where $\varepsilon\in (0,1)$ measures the size of the initial perturbation, $R\ge100$ is the scale parameter in the weighted energy,  $\delta\in(0,2/3)$ is the exponent related to  weights, and $N_*+1$ is the Sobolev regularity order assumed for the initial data.
 The pointwise estimate follows from the decay of the nonlocal pressure along nonlinear characteristics, while the Sobolev estimate combines higher-order weighted pressure bounds with estimates for derivatives of the characteristic flows. 
 Numerical experiments measure pointwise and Sobolev pressure-integral tails to verify the quadratic dependence of the initial perturbation amplitude.

\end{abstract}



\section{Introduction}

Alfv\'en waves, first predicted by Hannes Alfv\'en in 1942
\cite{Alfven1942}, are fundamental wave motions arising from the interactions between
electromagnetic fields and electrically conducting fluids. 
 These waves play an important role in plasma
physics and astrophysics, and their nonlinear dynamics are governed by
the magnetohydrodynamics (MHD) equations. We consider the ideal incompressible setting, in which both the fluid viscosity and
magnetic resistivity are absent while the velocity field is divergence free. In the presence of a strong constant magnetic background, Bardos,
Sulem and Sulem \cite{BardosSulemSulem1988} established global
existence for spatially localized solutions and studied their
long-time behaviour. A number of subsequent works have further
developed this line of research. For example, He, Xu and Yu
\cite{HeXuYu2018} employed a   weighted energy method to
construct the global nonlinear stability of Alfv\'en waves near
a strong constant magnetic field.

Let $v=v(t,x)$ be the fluid velocity, $b=b(t,x)$ the magnetic field, and $p=p(t,x)$ the total pressure (including the fluid pressure and magnetic pressure).  The ideal incompressible MHD system on $\mathbb{R}^{1+3}$ can be written as
\begin{equation}\label{eq:MHD}
	\begin{cases}
		\partial_t v+v\cdot\nabla v=-\nabla p+b\cdot\nabla b,\\
		\partial_t b+v\cdot\nabla b=b\cdot\nabla v,\\
		\operatorname{div} v=\operatorname{div} b=0.
	\end{cases}
\end{equation}
We employ the Els\"asser variables
\[
Z_+=v+b,\qquad Z_-=v-b,
\]
and study perturbations of the zero background velocity field and the strong constant background magnetic field
\[
B_0=e_3=(0,0,1).
\]
The strong magnetic  field has been normalized to unit strength, and the perturbation is assumed small relative to this background. Thus we use
\[
z_+=Z_+-B_0,\qquad z_-=Z_-+B_0
\]
to rewrite the system \eqref{eq:MHD} as
\begin{equation}\label{eq:elsasser}
	\begin{cases}
		L_-z_+=-\nabla p,\\
		L_+z_-=-\nabla p,\\
		\operatorname{div} z_+=\operatorname{div} z_-=0,
	\end{cases}
	\qquad
	L_\pm:=\partial_t+Z_\pm\cdot\nabla.
\end{equation}

Near a strong magnetic background, the two families of Alfv\'en waves propagate in
opposite directions along the background magnetic field. 
 For spatially
localized perturbations, the separation of these two families of Alfv\'en waves leads to time decay in the
nonlinear interaction terms, although the linearized waves themselves
are non-dispersive. In fact, the linearization of system \eqref{eq:elsasser} is
\[
(\partial_t-\partial_3)z_+=0,\qquad
(\partial_t+\partial_3)z_-=0,
\]
so the two wave families propagate in opposite $x_3$-directions without linear dispersion. Then the decay of nonlinear interactions, together with weighted energy estimates,
can allow  the construction of scattering fields at characteristic
infinities and   lead  to a scattering problem related to small-data Alfv\'en waves. 

For classical background on scattering theory, we refer the reader to \cite{LaxPhillips1967,Friedlander1980}. The corresponding inverse problem is to recover the initial data
from the scattering fields.  Rigidity from infinity results showed that the vanishing of scattering fields forces the initial Alfv\'en waves to vanish \cite{LiYu2021}.  Then the inverse scattering theory of \cite{Li2023} went further to establish local scattering
isomorphisms near the zero solution, so that  suitable pairs of small  scattering fields uniquely determine the initial Alfv\'en waves. 
In these works, the construction of the scattering fields relies on weighted energy
 estimates, the separation of two families of Alfv\'en waves, and the null structure inherent in the MHD system. In particular,
 the pressure is nonlocal and satisfies a Poisson equation whose source consists of
 products of derivatives from opposite wave families. This structure combines with the characteristic weights to yield the decay
 of $\nabla p$ that is integrable in time. The resulting convergence
 of the integrals in $\nabla p$ along nonlinear characteristics finally ensures the construction of 
  the scattering fields.

Our goal in the present paper is twofold.  First, we give a brief account of the  scattering mechanism, retaining only the parts needed to see how the weighted geometry, pressure estimate and scattering map fit together.  The global existence and scattering results are recalled from \cite{Li2023}, and the pressure gradient can be derived from its global weighted bounds. Second, we show one quantitative consequence  that  
the scattering fields can be approximated by observations made at a large finite time, with an explicit 
algebraic convergence rate in $L^\infty$ and in the (unweighted) Sobolev space $H^{N_*+1}$. The Sobolev estimate controls the error up to the same derivative order as the initial data.
Numerical experiments are also presented to support these quantitative estimates.

More precisely, we let $R\ge100$ and choose $\delta\in(0,2/3)$.  For global small solutions of size $\varepsilon$, the pressure estimate obtained from the weighted theory is
\begin{equation}\label{eq:intro-pressure}
 |\nabla p(t,x)|\le C\varepsilon^2(R+t)^{-(1+\delta)},\qquad t\ge0.
\end{equation}
If $z_{+,T}$ and $z_{-,T}$ denote the future finite time fields truncated at time $T$, while $z_+^{\infty}$ and $z_-^{\infty}$ denote the corresponding  scattering fields at future characteristic infinities $\mathcal{F}_+$ and $\mathcal{F}_-$,  the resulting quantitative estimate is
\begin{equation}\label{eq:intro-rate}
 \|z_+^{\infty}-z_{+,T}\|_{L^\infty(\mathcal{F}_+)}
 +\|z_-^{\infty}-z_{-,T}\|_{L^\infty(\mathcal{F}_-)}
 \le C_{\delta}\varepsilon^2(R+T)^{-\delta}.
\end{equation}
The proof in Theorem~\ref{thm:finite-time} will use the tail estimate obtained by integrating \eqref{eq:intro-pressure} along the nonlinear characteristics. This pointwise estimate gives a uniform bound for the
	error introduced by truncating the characteristic integrals
	defining the scattering fields. We see that the exponent is the rate naturally inherited from the weighted energy used in the small-data theory.  Higher-order weighted pressure estimates, together with control of
		derivatives of the nonlinear characteristic flows, yield a 
		strengthening estimate in Sobolev space:
		\begin{equation}\label{eq:intro-rate-sobolev}
			\|z_+^{\infty}-z_{+,T}\|_{H^{N_*+1}(\mathcal{F}_+)}
			+\|z_-^{\infty}-z_{-,T}\|_{H^{N_*+1}(\mathcal{F}_-)}
			\le C_{\delta}\varepsilon^2(R+T)^{-\delta}.
		\end{equation}
		Here, $N_*+1$ is the Sobolev regularity order assumed for the initial data, and the Sobolev norms are unweighted and taken with respect to the
		initial position of the characteristic flow.  Theorem~\ref{thm:finite-time-sobolev} will 
		prove this estimate under the same initial-data regularity assumptions
		as Theorem~A. 
To obtain an error bound for the reconstructed initial data
		using the inverse scattering theorem, one would also need to
		control the truncation error in the corresponding weighted
		Sobolev norms. 
		
	Previous related numerical investigations include the pseudospectral simulations of incompressible MHD turbulence by Maron and Goldreich~\cite{MaronGoldreich2001}. 
	Using gyrokinetic simulations, Nielson, Howes and Dorland~\cite{NielsonHowesDorland2013} tested weakly nonlinear predictions for Alfv\'en wave collisions,
	while Verniero, Howes and Klein~\cite{VernieroHowesKlein2018} studied localized wavepacket collisions and observed that nonlinear interactions occur during overlap and cease between collisions after separation.
	More recently, Shestov, Voitenko and Zhukov~\cite{ShestovVoitenkoZhukov2022} studied counterpropagating Alfv\'en-wave collisions numerically in compressible MHD. 
	These studies concern wave interactions and turbulent energy transfer. 
	Our numerical experiments focus on the pressure-integral tails that define the finite-time scattering error. 
	The computations combine a Fourier pseudospectral discretization with dealiasing \cite{Orszag1971},
	the Helmholtz--Leray projection \cite{Leray1934},
	and a fourth-order Runge--Kutta method \cite{Lawson1967}. 
We integrate the pressure contribution along nonlinear characteristics to finite reference times and measure both a sampled maximum and an unweighted $H^6$ norm of the same vector tails.
  The parameter variations in time step, spatial resolution, domain size, and reference time are considered.
  These computations exhibit the $\varepsilon^2$ dependence in \eqref{eq:intro-rate} and \eqref{eq:intro-rate-sobolev}.
	

The rest of this paper is organized as follows.  Section~2 recalls  the characteristic geometry, weighted energy estimates, and pressure decay estimates.  Section~3 describes the scattering fields and scattering map.  Section~4 is devoted to the two finite-time approximation theorems.  Section 5 presents several numerical tests.  Finally, Section~6 contains concluding remarks.

\section{Energy  and pressure estimate}

Let $u_\pm$ be characteristic functions satisfying
\begin{equation*}
 L_+u_+=0,\qquad L_-u_-=0,
 \qquad u_\pm|_{t=0}=x_3.
\end{equation*}
Similarly, the transverse labels $x_i^\pm$, $i=1,2$, are transported by $L_\pm$ and agree with $x_i$ at $t=0$.  The level sets of $u_+$ and $u_-$ form a double characteristic foliation.  We use the weights
\begin{equation*}
 \langle u_+\rangle=(R^2+|u_+|^2)^{1/2},\qquad
 \langle u_-\rangle=(R^2+|u_-|^2)^{1/2},
\end{equation*}
where $R\ge100$ and $0<\delta<2/3$.

\subsection{The weighted energy estimate}

The characteristic weight attached to $z_+$ is $\langle u_-\rangle^{1+\delta}$, whereas that attached to $z_-$ is $\langle u_+\rangle^{1+\delta}$.  At the basic level,  the weighted energies are defined as
\begin{equation*}
 E_\pm(t)=\big\|\langle u_\mp(t,\cdot)\rangle^{1+\delta}
 z_\pm(t,\cdot)\big\|_{L^2(\mathbb{R}^3)}^2,
\end{equation*}
together with higher-order analogues and fluxes through the characteristic hypersurfaces. Set
\[
 \mathcal{E}^{N_*}(t):=\sum_{\pm}\sum_{k=0}^{N_*+1}
 \big\|\langle u_\mp(t,\cdot)\rangle^{1+\delta}
            \nabla^k z_\pm(t,\cdot)\big\|_{L^2(\mathbb{R}^3)}^2,
\]
where $\nabla^k$ denotes the collection of Cartesian spatial derivatives of order $k$. Throughout this paper, $R$, $\delta$, and the integer $N_*$ are fixed; constants may depend on them but are uniform in time and in the small-data class specified below.
\par

A convenient version of the global small-data theorem reads as follows.

\begin{theoremA}[Small-data global existence; \cite{HeXuYu2018,Li2023}]
Let $R\ge100$, $\delta\in(0,2/3)$ and $N_*\ge5$.  There exists $\varepsilon_0>0$ such that if the divergence-free initial data satisfy
\begin{equation*}
 \mathcal{E}^{N_*}(0):=
 \sum_{\pm}\sum_{k=0}^{N_*+1}
 \big\|(R^2+|x_3|^2)^{(1+\delta)/2}\nabla^kz_\pm(0)\big\|_{L^2(\mathbb{R}^3)}^2
 \le \varepsilon^2,
\end{equation*}
with $0<\varepsilon\le\varepsilon_0$, then the system \eqref{eq:elsasser} admits a unique global solution.   Its weighted energies and characteristic fluxes up to the corresponding orders remain bounded by $C\mathcal{E}^{N_*}(0)$. In particular, the weighted energies satisfy \[\sup_{t\ge0}\mathcal{E}^{N_*}(t)\le C\mathcal{E}^{N_*}(0).\]
\end{theoremA}

Two consequences of these estimates will drive the  scattering argument in the following sections.  

\begin{lemma}\label{lemma:separation} Under the assumptions of Theorem~A, we have the following separation property:
\begin{equation*}
	\langle u_+\rangle\langle u_-\rangle\ge c(R+t),\qquad t\ge0.
\end{equation*}
\end{lemma} 
\begin{proof}
After reducing $\varepsilon_0$ so that
$\|z_\pm\|_{L^\infty_{t,x}}\le1/4$, the characteristic equations 
together with the initial conditions $u_\pm(0,x)=x_3$ imply 
\[
 |u_+(t,x)-(x_3-t)|\le t/4,\qquad
 |u_-(t,x)-(x_3+t)|\le t/4.
\]
Thus $|u_--u_+|\ge3t/2$. By the triangle inequality, it follows that
$\max\{|u_-|,|u_+|\}
\ge \frac12|u_--u_+|
\ge \frac{3t}{4}$.
Since both weights are at least $R$ and at least one is also
at least $3t/4$, we obtain
\[
\langle u_+\rangle\langle u_-\rangle
\ge R\max\{R,3t/4\}
\ge \frac{R}{2}\left(R+\frac{3t}{4}\right)
\ge \frac38(R+t),
\]
where the last inequality uses $R\ge1$.
This proves the desired estimate with $c=3/8$.
\end{proof}

	\begin{lemma}\label{Lemma:Sobolev}
		Under the assumptions of Theorem~A, for each integer
		$0\le j\le N_*-1$, we have
		\[
			\big\|\langle u_\mp\rangle^{1+\delta}\nabla^jz_\pm\big\|_{L^\infty}
		\le C\varepsilon.
		\]
	\end{lemma}
	
	\begin{proof}
			Since $\langle u_\pm\rangle=(R^2+|u_\pm|^2)^{1/2}\ge R\ge1$,
		the chain rule  yields
		\[
		\big|\nabla^\ell\langle u_\pm\rangle^{1+\delta}\big|
		\le C\langle u_\pm\rangle^{1+\delta},
		\qquad 0\le\ell\le2.
		\]
By definition of weights and geometric estimates for the characteristic functions, it is also clear to see that
		\[
		\sup_{t\ge0}
		\left(
		\|\nabla u_\pm(t)\|_{L^\infty(\mathbb{R}^3)}
		+\|\nabla^2u_\pm(t)\|_{L^\infty(\mathbb{R}^3)}
		\right)\le C.
		\]
	
		Fix $t\ge0$ and an integer $0\le j\le N_*-1$.
		The Sobolev embedding
		$H^2(\mathbb{R}^3)\hookrightarrow L^\infty(\mathbb{R}^3)$ and the Leibniz rule
		imply
		\begin{align*}
			\big\|\langle u_\mp\rangle^{1+\delta}\nabla^jz_\pm\big\|_{L^\infty}
			&\le C\big\|\langle u_\mp\rangle^{1+\delta}\nabla^jz_\pm\big\|_{H^2}\\
			&\le C\sum_{\ell=0}^{2}
			\big\|\langle u_\mp\rangle^{1+\delta}\nabla^{j+\ell}z_\pm\big\|_{L^2}\\
			&\le C\big(\mathcal{E}^{N_*}(t)\big)^{1/2}
			\le C\varepsilon.
		\end{align*}
		Here all norms are taken over $\mathbb{R}^3$, and the last two inequalities
		follow from $j+2\le N_*+1$ and Theorem~A.
	\end{proof}
	
	Combining Lemmas~\ref{lemma:separation} and~\ref{Lemma:Sobolev}
	yields the following interaction estimate.
	\begin{lemma}
		Under the assumptions of Theorem~A, for any integers
		$0\le a,b\le N_*-1$ and all $(t,x)\in[0,\infty)\times\mathbb{R}^3$, we have
		\begin{equation}\label{eq:product-decay}
			|\nabla^a z_+(t,x)|\,|\nabla^b z_-(t,x)|
			\le C\varepsilon^2(R+t)^{-(1+\delta)}.
		\end{equation}
	\end{lemma}
	
	\begin{proof}
		By Lemma~\ref{Lemma:Sobolev}, we have
		\[
		|\nabla^a z_+|\,|\nabla^b z_-|
		\le C\varepsilon^2\langle u_-\rangle^{-(1+\delta)}
		\langle u_+\rangle^{-(1+\delta)}
		=C\varepsilon^2\big(\langle u_-\rangle\langle u_+\rangle\big)^{-(1+\delta)}.
		\]
	Using	Lemma~\ref{lemma:separation} then gives \eqref{eq:product-decay}.
	\end{proof}

From this lemma, we notice the essential feature that decay is gained when the two opposite wave families interact, rather than from dispersion of either family separately.

\subsection{The pressure decay}

The pressure is the main nonlocal term in \eqref{eq:elsasser}.  Taking the divergence and using incompressibility gives the Poisson equation
\begin{equation}\label{eq:pressure-poisson}
	-\Delta p=\partial_i z_-^j\,\partial_j z_+^i
	=\partial_i\partial_j(z_-^jz_+^i).
\end{equation}
The right-hand side has a favorable structure that every product contains one factor from each family.  

In view of \eqref{eq:product-decay}, this source already contains the time decay generated by the separation of the two families. 
We use the Newtonian representation to give a pointwise estimate for $\nabla p$, which   is part of Lemma 3.15 in \cite{Li2023} and is sufficient
for Theorem~\ref{thm:finite-time}. 
 
\begin{lemma}\label{lem:pressure-gradient}
	Under the assumptions of Theorem~A, we have the following pressure-gradient decay:
	\begin{equation*}
		\|\nabla p(t)\|_{L^\infty(\mathbb{R}^3)}
		\le C\varepsilon^2(R+t)^{-1-\delta},\qquad t\ge0.
	\end{equation*}
\end{lemma}

\begin{proof}
	Fix $t\ge0$. Now we choose a cut-off function $\chi\in C_c^\infty(\mathbb{R}^3)$ with $0\le\chi\le1$, $\chi=1$ on $|x|\le1$, and $\chi=0$ on $|x|\ge2$. Using the Newtonian representation of \eqref{eq:pressure-poisson}, we split the integral into near-field and far-field parts. Integrating by parts twice in the far-field part gives
	\begin{equation*}
		\begin{aligned}
			\nabla p(t,x)
			={}&-\frac{1}{4\pi}\sum_{i,j=1}^3\int_{\mathbb{R}^3}
			\chi(x-x')\frac{x-x'}{|x-x'|^3}
			\partial_i z_-^j(t,x')\,\partial_j z_+^i(t,x')\,dx'\\
			&-\frac{1}{4\pi}\sum_{i,j=1}^3\int_{\mathbb{R}^3}
			\partial_{x'_i}\partial_{x'_j}
			\left((1-\chi(x-x'))\frac{x-x'}{|x-x'|^3}\right)
			z_-^j(t,x')z_+^i(t,x')\,dx'.
		\end{aligned}
	\end{equation*}
	This identity follows first for smooth rapidly decaying fields and then by approximation. Moreover, the solution regularity and the $L^2$ bounds in Theorem~A ensure that the products appearing here have the required integrability.
	
	It is easy to see that
	the near-field kernel is bounded by $C|x-x'|^{-2}$ and vanishes for $|x-x'|\ge2$; the far-field kernel is bounded by $C|x-x'|^{-4}$ and vanishes for $|x-x'|\le1$; while terms involving derivatives of $\chi$ are supported in $1\le|x-x'|\le2$. Thus, by \eqref{eq:product-decay}, we derive
	\begin{align*}
		|\nabla p(t,x)|
		&\le C\int_{|x-x'|\le2}
		\frac{|\nabla z_-(t,x')|\,|\nabla z_+(t,x')|}{|x-x'|^2}\,dx' +C\int_{|x-x'|\ge1}
		\frac{|z_-(t,x')|\,|z_+(t,x')|}{|x-x'|^4}\,dx'\\
		&\le C\varepsilon^2(R+t)^{-1-\delta}
		\left(\int_{|x-x'|\le2}\frac{dx'}{|x-x'|^2}
		+\int_{|x-x'|\ge1}\frac{dx'}{|x-x'|^4}\right)\\
		&\le C\varepsilon^2(R+t)^{-1-\delta},
	\end{align*}
	since both integrals are finite in dimension three. Taking the supremum over $x\in\mathbb{R}^3$ proves the lemma.
\end{proof}

\begin{remark}
 The pressure-gradient decay in Lemma~\ref{lem:pressure-gradient} is one of the central manifestations of the MHD null structure in the weighted framework.  Since $\delta>0$, the first derivative is integrable in time.  Together with the uniform smallness of $z_\pm$, it immediately implies convergence of integrals in  \eqref{eq:scatter-plus}--\eqref{eq:scatter-minus}.
\end{remark}

 The higher-order weighted pressure estimate needed for Theorem~\ref{thm:finite-time-sobolev} is provided as follows.
 
	\begin{lemma}[Higher-order weighted pressure estimates]
		\label{lem:pressure-higher}
		Under the assumptions of Theorem~A,  we have the following decay estimate for higher-order derivatives of pressure:
		\begin{equation}\label{eq:pressure-higher-weighted}
			\sum_{\pm}\sum_{j=0}^{N_*+1}
			\int_0^\infty\!\!\int_{\mathbb{R}^3}
			\langle u_\mp\rangle^{2(1+\delta)}\langle u_\pm\rangle^{1+\delta}
			|\nabla^j\nabla p|^2\,dx\,dt
			\le C\varepsilon^4.
		\end{equation}
		In particular,
		\begin{equation}\label{eq:pressure-time-hm}
			\int_0^\infty(R+t)^{2(1+\delta)}
			\|\nabla p(t)\|_{H^{N_*+1}(\mathbb{R}^3)}^2\,dt
			\le C\varepsilon^4.
		\end{equation}
		Here the Sobolev norm is unweighted and taken with respect to the
		Cartesian spatial variable $x$, and $C$ depends only on the fixed
		parameters $R$, $\delta$ and $N_*$.
	\end{lemma}
	
	\begin{proof}
		The weighted spacetime pressure estimates, i.e. Lemma 3.20 in \cite{Li2023}, directly give \eqref{eq:pressure-higher-weighted}. They control
		spatial derivatives of $\nabla p$ of orders $0$ through $N_*+1$ using the
		weighted energy and flux bounds in Theorem~A. Thus the highest index
		$j=N_*+1$ in \eqref{eq:pressure-higher-weighted} requires no
		additional regularity of the initial data.
		
	The proof of Lemma~\ref{lemma:separation} gives
	$|u_--u_+|\ge 3t/2$. Since
	$\langle u_\pm\rangle\ge R$ and $\langle u_\pm\rangle\ge |u_\pm|$,
	we obtain
	\[
	\max\{\langle u_+\rangle,\langle u_-\rangle\}
	\ge \max\left\{R,\frac12|u_--u_+|\right\}
	\ge \max\left\{R,\frac34t\right\}
	\ge \frac38(R+t).
	\]
	Since $\langle u_\pm\rangle\ge R\ge1$, it follows that
	\begin{align*}
		\langle u_-\rangle^{2(1+\delta)}\langle u_+\rangle^{1+\delta}
		+\langle u_+\rangle^{2(1+\delta)}\langle u_-\rangle^{1+\delta}
		&\ge \max\{\langle u_+\rangle,\langle u_-\rangle\}^{2(1+\delta)}
		\ge \left(\frac38\right)^{2(1+\delta)}(R+t)^{2(1+\delta)}.
	\end{align*}
	Multiplying this inequality by $|\nabla^j\nabla p|^2$,
	summing over $0\le j\le m$, and integrating in space and time,
	we obtain \eqref{eq:pressure-time-hm}
	from \eqref{eq:pressure-higher-weighted}. The proof of this lemma is now complete.
	\end{proof}

\section{Construction and analysis of scattering fields}

We parameterize the characteristic flows by the physical time $t$,
with initial position $\xi\in\mathbb{R}^3$:
\begin{equation*}
	\partial_tX_\pm(t,\xi)=Z_\pm(t,X_\pm(t,\xi)),
	\qquad X_\pm(0,\xi)=\xi.
\end{equation*}
For each initial label $\xi_-=(x_1^-,x_2^-,u_-)$, let $l_-(\xi_-)$ denote the integral curve of $L_-$ carrying that label.  Likewise $l_+(\xi_+)$, with $\xi_+=(x_1^+,x_2^+,u_+)$, is an integral curve of $L_+$.  
In these global characteristic coordinates, the transport equations in \eqref{eq:elsasser} can be integrated along the corresponding curves. Their future endpoints define the characteristic infinities
\[
\mathcal{F}_+=\{l_-(\xi_-):t=\infty\},\qquad
\mathcal{F}_-=\{l_+(\xi_+):t=\infty\},
\]
and the past endpoints similarly define $\mathcal{P}_+$ and $\mathcal{P}_-$.   By symmetry of time, we mainly focus on $\mathcal{F}_+$ and $\mathcal{F}_-$ in this paper.
The spaces $\mathcal{F}_\pm$ are identified with $\mathbb{R}^3$ by their initial labels, and $L^\infty(\mathcal{F}_\pm)$ denotes the essential supremum over these labels.

We retain the line-measure convention used in \cite{Li2023}. Thus the future scattering fields with this line-integral convention are defined by
\begin{align}
 z_+^{\infty}(\xi_-)
 &=z_+(0,\xi_-)-\int_0^\infty
 \nabla p(\tau,\xi_-)
 \sqrt{1+|z_-(\tau,\xi_-)-B_0|^2}\,d\tau,
 \label{eq:scatter-plus}\\
 z_-^{\infty}(\xi_+)
 &=z_-(0,\xi_+)-\int_0^\infty
 \nabla p(\tau,\xi_+)
 \sqrt{1+|z_+(\tau,\xi_+)+B_0|^2}\,d\tau.
 \label{eq:scatter-minus}
\end{align}
Here and below, spacetime quantities expressed in terms of
$\xi_\pm$ are evaluated along the corresponding characteristic
curves. For example, $\nabla p(t,\xi_\pm)$ denotes
$\nabla p(t,X_\pm(t,\xi_\pm))$.  Past scattering fields are defined analogously by integrating toward $t=-\infty$. By symmetry of time, we mainly focus on future scattering fields in this paper. 

Theorems~2.2, 2.3 and 2.4 in \cite{Li2023} show, respectively,
that these integrals converge, that the scattering fields belong
to the natural weighted Sobolev spaces on the characteristic
infinities, and that their difference from the corresponding
linear propagating fields is of quadratic order in the size of the initial data. Moreover, the future scattering fields $z_\pm^{\infty}$ differ from the corresponding initial profiles $z_\pm(0)$ only by a quadratic correction of size \(O(\varepsilon^2)\). 
This $O(\varepsilon^2)$ deviation shows that to first order in the size
of the initial data, the nonlinear scattering map agrees with the linear propagation map at the zero solution.

 One important application of scattering fields is to define  scattering maps. 
The four characteristic infinities allow four natural choices of paired scattering data: future--future, past--future, past--past, and future--past.  Each pair defines a nonlinear scattering map from small initial data to two scattering fields.  If $\mathcal{N}^{(\#)}$ denotes any of these nonlinear maps and $\mathcal{L}^{(\#)}$ denotes the corresponding linear propagation map, the  inverse scattering result in \cite{Li2023} can be briefly stated as
\begin{equation*}
 d\mathcal{N}^{(\#)}\big|_{0}=\mathcal{L}^{(\#)},
\end{equation*}
where the linear maps are isomorphisms between the corresponding weighted Sobolev spaces, the nonlinear scattering maps below are considered between the weighted
Sobolev spaces of admissible initial data and scattering fields. 

\begin{remark} 
We further remark that an application of inverse function theorem can yield the following inverse scattering statement: 
For sufficiently small weighted initial data as in Theorem~A, each of the four nonlinear scattering maps described above is a local diffeomorphism at the zero solution  in the admissible weighted spaces.  Consequently, the associated pair of scattering fields determines the initial pair $(z_+(0),z_-(0))$ uniquely in a neighborhood of zero. We refer the readers to Theorem 2.5 in \cite{Li2023} for a complete and  detailed version of inverse scattering result for Alfv\'en waves. 
\end{remark}

\section{Finite-time approximation of the future scattering fields}

In this section, we  estimate the error between the finite-time characteristic observation fields and their infinite-time scattering fields. 
We first derive a pointwise approximation estimate by integrating
the pressure decay along nonlinear characteristics. Then we strengthen
this estimate to the unweighted space $H^{N_*+1}$, using higher-order
pressure bounds and estimates for derivatives of the characteristic flows.

For $T\ge0$, we define the finite-time future profiles by cutting the  integrals in future scattering fields at time $T$:
\begin{align}
 z_{+,T}(\xi_-)
 &: =z_+(0,\xi_-)-\int_0^T
 \nabla p(\tau,\xi_-)
 \sqrt{1+|z_-(\tau,\xi_-)-B_0|^2}\,d\tau,
 \label{eq:trunc-plus}\\
 z_{-,T}(\xi_+)
 &: =z_-(0,\xi_+)-\int_0^T
 \nabla p(\tau,\xi_+)
 \sqrt{1+|z_+(\tau,\xi_+)+B_0|^2}\,d\tau.
 \label{eq:trunc-minus}
\end{align}
Indeed, these formulas are the finite-time counterparts of \eqref{eq:scatter-plus}--\eqref{eq:scatter-minus}, expressed in the same nonlinear characteristic coordinates.

\begin{theorem}[Pointwise finite-time approximation]\label{thm:finite-time}
Assume the hypotheses of Theorem~A.  Let $z_+^{\infty}$ on $\mathcal{F}_+$ and $z_-^{\infty}$ on $\mathcal{F}_-$ be the future scattering fields defined by \eqref{eq:scatter-plus}--\eqref{eq:scatter-minus}, and let $z_{+,T}$ and $z_{-,T}$ be defined by \eqref{eq:trunc-plus}--\eqref{eq:trunc-minus}.  Then there exists a constant $C_\delta$, independent of $T$ and uniform for $\varepsilon$, with allowed dependence on the fixed $R$, $\delta$ and $N_*$, such that for all $T\ge0$,
\begin{equation}\label{eq:finite-rate}
 \|z_+^{\infty}-z_{+,T}\|_{L^\infty(\mathcal{F}_+)}
 +\|z_-^{\infty}-z_{-,T}\|_{L^\infty(\mathcal{F}_-)}
 \le C_\delta\varepsilon^2(R+T)^{-\delta}.
\end{equation}
In particular, the future scattering fields are approximated by finite-time characteristic observation fields with an explicit error bound. 
\end{theorem}

\begin{proof}
We only prove the estimate for $z_+$ since the other family is identical.  Subtracting \eqref{eq:trunc-plus} from \eqref{eq:scatter-plus} gives the tail identity
\begin{equation}\label{eq:tail}
 z_+^{\infty}(\xi_-)-z_{+,T}(\xi_-)
 =-\int_T^\infty
 \nabla p(\tau,\xi_-)
 \sqrt{1+|z_-(\tau,\xi_-)-B_0|^2}\,d\tau.
\end{equation}
By the weighted Sobolev estimate in Lemma \ref{Lemma:Sobolev}, $z_-$ is uniformly $O(\varepsilon)$.  After reducing $\varepsilon_0$ if necessary, the square-root factor in \eqref{eq:tail} is therefore bounded by a universal constant.  The pressure estimate  in Lemma~\ref{lem:pressure-gradient} yields
\[
 \big|\nabla p(\tau,\xi_-)
 \sqrt{1+|z_-(\tau,\xi_-)-B_0|^2}\big|
 \le C\varepsilon^2(R+\tau)^{-1-\delta}.
\]
Hence, uniformly in $\xi_-$, we obtain
\begin{align*}
 |z_+^{\infty}(\xi_-)-z_{+,T}(\xi_-)|
 &\le C\varepsilon^2\int_T^\infty(R+\tau)^{-1-\delta}\,d\tau=\frac{C}{\delta}\varepsilon^2(R+T)^{-\delta}.
\end{align*}
Taking the supremum over $\xi_-$ gives the estimate on $\mathcal{F}_+$.
The same argument along the $L_+$ characteristics gives the
estimate on $\mathcal{F}_-$. Adding the two estimates and absorbing
the resulting constant into $C_\delta$ proves
\eqref{eq:finite-rate}. This finally completes the proof of the theorem.
\end{proof}

We now extend the estimate in Theorem~\ref{thm:finite-time} to
unweighted Sobolev norms on the characteristic infinities $\mathcal{F}_\pm$.
We use the initial positions $\xi_\mp\in\mathbb{R}^3$ of the characteristic
curves as coordinates on $\mathcal{F}_\pm$. For each nonnegative integer $m$,
we define
\[
\|f\|_{H^m(\mathcal{F}_\pm)}^2
:=\sum_{|\alpha|\le m}\int_{\mathbb{R}^3}
|\partial_{\xi_\mp}^{\alpha}f(\xi_\mp)|^2\,d\xi_\mp.
\]
The argument below controls derivatives with respect to $\xi_\mp$
up to the same order as the initial data.

	\begin{theorem}[Sobolev finite-time approximation]
		\label{thm:finite-time-sobolev}
		Assume the hypotheses of Theorem~A, and let
		$z_\pm^\infty$ and $z_{\pm,T}$ be the scattering fields and
		finite-time profiles defined above. After reducing $\varepsilon_0$ if
		necessary, there exists a constant $C_\delta$, depending only on
		the fixed parameters $R$, $\delta$ and $N_*$, such that
		\begin{equation}\label{eq:finite-rate-sobolev}
			\|z_+^\infty-z_{+,T}\|_{H^{N_*+1}(\mathcal{F}_+)}
			+\|z_-^\infty-z_{-,T}\|_{H^{N_*+1}(\mathcal{F}_-)}
			\le C_\delta\varepsilon^2(R+T)^{-\delta},\qquad T\ge0.
		\end{equation}
	\end{theorem}
	
	\begin{proof} 	We only prove the estimate for $z_+$, since the estimate for $z_-$ is analogous. 
 Throughout the proof,
		$D_\xi^j$ denotes the collection of
	 derivatives with respect to the initial position $\xi$ 
		of order $j$, whereas $\nabla^j$ denotes Cartesian spatial derivatives.

	We  note that the characteristic geometry estimates (see Lemma 3.9 in \cite{Li2023}) 
	give $\nabla u_+=e_3+O(\varepsilon)$ uniformly. Together with
		Lemma~\ref{Lemma:Sobolev}, this yields
		\[
		\frac{d}{dt}u_+(t,X_-(t,\xi))
		=\big((Z_--Z_+)\cdot\nabla u_+\big)(t,X_-(t,\xi))
		=-2+O(\varepsilon).
		\]
		For sufficiently small $\varepsilon_0$, this derivative lies between
		$-3$ and $-1$. Changing variables along each characteristic gives
		\begin{equation}\label{eq:opposite-weight-integral-hm}
			\sup_{t\ge0,\,\xi\in\mathbb{R}^3}
			\int_0^t
			\langle u_+(s,X_-(s,\xi))\rangle^{-(1+\delta)}\,ds
			\le C_\delta.
		\end{equation}
		Differentiating the flow equation and applying
		Lemma~\ref{Lemma:Sobolev},
		\eqref{eq:opposite-weight-integral-hm}, and Gronwall's inequality
		inductively, we obtain
		\begin{equation}\label{eq:flow-low-hm}
			\sup_{t\ge0}\left(
			\|D_\xi X_-(t)-I\|_{L^\infty}
			+\sum_{j=2}^{N_*-1}
			\|D_\xi^jX_-(t)\|_{L^\infty}
			\right)\le C\varepsilon.
		\end{equation}

			For the two highest orders, the differentiated flow equations
			form a linear system in $D_\xi^{N_*}X_-$ and
			$D_\xi^{N_*+1}X_-$. By 
			Lemma~\ref{Lemma:Sobolev}, 
			\eqref{eq:opposite-weight-integral-hm} and \eqref{eq:flow-low-hm}, the coefficient
			integrals are bounded by $C\varepsilon$ uniformly in $t$ and $\xi$.
			Since these flow derivatives vanish at $t=0$,
			Gronwall's inequality yields
			\[
			|D_\xi^{N_*}X_-(t,\xi)|
			+|D_\xi^{N_*+1}X_-(t,\xi)|
			\le C\int_0^t\sum_{j=2}^{N_*+1}
			|\nabla^jz_-(s,X_-(s,\xi))|\,ds.
			\]
			Since $\operatorname{div} Z_-=0$, the flow $X_-$ preserves volume, i.e.  we have
			$\det D_\xi X_-(t,\xi)=1$.
			Moreover, due to $\langle u_+\rangle\ge1$,  \eqref{eq:opposite-weight-integral-hm} also holds
			with exponent $2(1+\delta)$.
			Thus weighted Cauchy--Schwarz inequality
			and Theorem~A give
			\begin{equation}\label{eq:flow-high-hm}
					\|D_\xi^{N_*}X_-(t)\|_{L^2}
					+\|D_\xi^{N_*+1}X_-(t)\|_{L^2}
				 \le C\left(
					\int_0^t\sum_{j=2}^{N_*+1}
					\|\langle u_+(s,\cdot)\rangle^{1+\delta}
					\nabla^jz_-(s)\|_{L^2}^2\,ds
					\right)^{1/2}
					\le C\varepsilon\sqrt t.
			\end{equation}

		We next apply the chain rule to composition with $X_-(t,\cdot)$.
		Terms containing only flow derivatives of order at most $N_*-1$
		are controlled by \eqref{eq:flow-low-hm} and volume preservation.
		A term containing a flow derivative of order $N_*$ or $N_*+1$
		contains at most two derivatives of the composed function.
		Thus \eqref{eq:flow-high-hm} gives, for $f\in H^{N_*+1}(\mathbb{R}^3)$,
		\begin{align}
			\|f\circ X_-(t,\cdot)\|_{H^{N_*+1}_\xi}
			&\le C\|f\|_{H^{N_*+1}_x}
			+C\varepsilon\sqrt t\,\|f\|_{W^{2,\infty}_x} \le C(1+\varepsilon\sqrt t)\|f\|_{H^{N_*+1}_x},
			\label{eq:pullback-hm}
		\end{align}
		where we used $H^{N_*+1}(\mathbb{R}^3)\hookrightarrow W^{2,\infty}(\mathbb{R}^3)$.
Since	Theorem~A and  Sobolev composition estimate give
		\[
		\left\|\sqrt{1+|Z_-(t)|^2}-\sqrt2\right\|_{H^{N_*+1}_x}
		\le C\|z_-(t)\|_{H^{N_*+1}_x}\le C\varepsilon,
		\]
then	Sobolev product   estimate further implies
		\[
		\left\|\nabla p(t)\sqrt{1+|Z_-(t)|^2}\right\|_{H^{N_*+1}_x}
		\le C\|\nabla p(t)\|_{H^{N_*+1}_x}.
		\]
		Combining this with \eqref{eq:pullback-hm}, we find
		\[
		\left\|
		\left(\nabla p\sqrt{1+|Z_-|^2}\right)
		(t,X_-(t,\cdot))
		\right\|_{H^{N_*+1}_\xi}
		\le C(1+\varepsilon\sqrt t)\|\nabla p(t)\|_{H^{N_*+1}_x}.
		\]	Thus the tail in \eqref{eq:tail} converges absolutely in $H^{N_*+1}$.

By Lemma~\ref{lem:pressure-higher},
		Minkowski's inequality and Cauchy--Schwarz inequality, we infer that
		\begin{align*}
			\|z_+^\infty-z_{+,T}\|_{H^{N_*+1}(\mathcal{F}_+)}
			&\le C\int_T^\infty
			(1+\varepsilon\sqrt t)\|\nabla p(t)\|_{H^{N_*+1}_x}\,dt\\
			&\le C\left[
			\left(\int_T^\infty(R+t)^{-2(1+\delta)}\,dt\right)^{1/2}
			+\varepsilon\left(\int_T^\infty t(R+t)^{-2(1+\delta)}\,dt\right)^{1/2}
			\right]\\
			&\qquad\times
			\left(\int_T^\infty(R+t)^{2(1+\delta)}
			\|\nabla p(t)\|_{H^{N_*+1}_x}^2\,dt\right)^{1/2}\\
			&\le C_\delta\varepsilon^2
			\left((R+T)^{-\frac12-\delta}
			+\varepsilon(R+T)^{-\delta}\right)\\
			&\le C_\delta\varepsilon^2(R+T)^{-\delta},
		\end{align*}
where  we have used
		\[
		\int_T^\infty(R+t)^{-2(1+\delta)}\,dt
		=\frac{(R+T)^{-1-2\delta}}{1+2\delta},
		\qquad
		\int_T^\infty t(R+t)^{-2(1+\delta)}\,dt
		\le\frac{(R+T)^{-2\delta}}{2\delta},
		\]
		as well as $\varepsilon\in (0,1)$ and $R+T\ge1$. 	The estimate for $z_-$ follows analogously along $X_+$.
		Adding the two bounds finally proves \eqref{eq:finite-rate-sobolev}. 
\end{proof}

\begin{remark}[Relation to the existing scattering theory]
The existence and weighted Sobolev regularity of $z_\pm^{\infty}$,  the pressure estimate in  Lemma~\ref{lem:pressure-gradient}, and the higher-order weighted pressure estimates in Lemma~\ref{lem:pressure-higher} belong to the established  theory in \cite{Li2023}.  Theorems~\ref{thm:finite-time} and~\ref{thm:finite-time-sobolev} isolate their quantitative  finite-time approximation consequences: instead of only taking $T\to\infty$, they record the explicit tail rate $(R+T)^{-\delta}$.  
Moreover, the Sobolev estimate retains the initial derivative order in an unweighted norm. No additional structural assumption on the solution is introduced.
\end{remark}

\begin{remark}[Rate inherited from the weight]
The exponent $\delta$ is determined by the choice of $1+\delta$ in the weighted energy.  
For the pointwise estimate, any improvement of the integrable pressure decay immediately improves the finite-time approximation rate by integration of the tail. For the Sobolev estimate, the rate also depends on the control of high-order derivatives of the characteristic flows.
\end{remark}

\section{Numerical validation}\label{sec:numerics}
We examine the quadratic amplitude dependence in Theorems~\ref{thm:finite-time} and~\ref{thm:finite-time-sobolev} using two norms of the same pressure-integral tails.
The computations use a finite periodic domain and finite reference times.
They test the $\varepsilon^2$ dependence and check that the observed behaviour is stable under parameter variation.

\subsection{Numerical method and parameters}\label{sec:numerical-setup}
Let $\sigma=2.5$, $d=(0.7, -0.4, 0.3)$, and
\begin{equation*}
 \psi_+(x)=e^{-|x|^2/\sigma^2},~
 \psi_-(x)=e^{-|x-d|^2/\sigma^2},~
 f_+=\nabla\times(0,\psi_+,0),~
 f_-=\nabla\times(\psi_-,0,0).
\end{equation*}
With $R=100$, $\delta=0.5$ and $N_*=5$, define the whole-space normalization by
\begin{equation*}
 K^2=\sum_\pm\sum_{k=0}^{6}
 \big\|(R^2+x_3^2)^{(1+\delta)/2}\nabla^k f_\pm\big\|_{L^2(\mathbb{R}^3)}^2,
 \quad
 |\nabla^k f|^2=\sum_{|\alpha|=k}\frac{k!}{\alpha!}|\partial^\alpha f|^2.
\end{equation*}
Here $\alpha=(\alpha_1,\alpha_2,\alpha_3)\in\mathbb{N}_0^3$ is a multiindex, with $|\alpha|=\alpha_1+\alpha_2+\alpha_3$ and $\alpha!=\alpha_1!\alpha_2!\alpha_3!$.
Thus the localized data $\varepsilon f_\pm/K$ have weighted energy
$\mathcal{E}^5(0)=\varepsilon^2$, with $K\approx20965$.
For the computation, we make the Gaussian potentials periodic over $\Omega=\prod_{j=1}^3[-L_j/2,L_j/2)$,
take their curls, and multiply by $\varepsilon/K$,
then the resulting initial fields are periodic and divergence-free.

We use a Fourier pseudospectral method with strict $2/3$ dealiasing,
retaining $3|n_j|<N_j$~\cite{Orszag1971}.
Here $N_j$ is the number of grid points and $n_j$ is the integer Fourier mode index in direction $j$.
At every stage the Helmholtz--Leray projection~\cite{Leray1934,Orszag1971} is used to enforce the divergence-free constraint,
while the longitudinal part of the nonlinear term gives the pressure gradient.
Background translations are integrated exactly in Fourier space, and a fourth-order Runge--Kutta method~\cite{Lawson1967} advances the nonlinear fields, characteristic displacements and pressure integrals.
The domain lengths $L=(L_1,L_2,L_3)$, physical grid sizes $N=(N_1,N_2,N_3)$, grid spacings, time-step size, and amplitude are
\begin{align*}
	&L=(32,32,48),\quad N=(64,64,96),\quad
	h_x=h_y=h_z=0.5,\\
	&\Delta t=0.05,\quad\varepsilon\in\{0.4, 0.2, 0.1, 0.05, 0.025\}.
\end{align*}
We integrate the characteristic trajectories $X_\mp(t;\xi)$, with $X_\mp(0;\xi)=\xi$, from a uniform periodic grid $\Gamma_{\rm int}$ of initial labels. This grid has $M_{\rm int}=64\times64\times96$ points and spacing $h_{\rm int}=0.5$ in each direction.
For each parameter setting, both norms are obtained from one shared evolution of the fields and characteristics.
We observe the tails at $T=0,1,2,3,4,5,6$ and use $T_{\texttt{ref}}=10$ as the primary reference. Each evolution continues to $t=12$, allowing reference times $10$ and $12$ to be compared.
Define
\begin{equation}\label{eq:numerical-sobolev-tail}
	 Q_{\pm,T}^{\varepsilon,T_{\texttt{ref}}}(\xi) := \int_T^{T_{\texttt{ref}}}\nabla p^\varepsilon\bigl(t,X_\mp(t;\xi)\bigr)\sqrt{1+|\dot X_\mp(t;\xi)|^2}\,dt.
\end{equation}
By \eqref{eq:trunc-plus}--\eqref{eq:trunc-minus},
$Q_{\pm,T}^{\varepsilon,T_{\texttt{ref}}} = z_{\pm,T}^{\varepsilon}-z_{\pm,T_{\texttt{ref}}}^{\varepsilon}$.
Here $p^\varepsilon$ is the pressure corresponding to the initial amplitude $\varepsilon$. After quadrature, we form the Fourier interpolant $Q_{h,\pm,T}^{\varepsilon,T_{\texttt{ref}}}$ of each vector tail from its values on $\Gamma_{\rm int}$.

\subsection{Pointwise tail estimate}\label{sec:pointwise-numerics}
For the tails in \eqref{eq:numerical-sobolev-tail}, let $\Gamma_h$ be a uniform measurement grid in the domain $\Omega$, with $M_{\rm eval}=496\times496\times744$ sample points. 
We define the sampled pointwise tail error by
\begin{equation*}
 E_{\infty,h}^{\varepsilon}(T,T_{\texttt{ref}})
 :=\sum_\pm\max_{\xi\in\Gamma_h}
       |Q_{h,\pm,T}^{\varepsilon,T_{\texttt{ref}}}(\xi)|,
\end{equation*}
Figure~\ref{fig:order} plots this quantity against $\varepsilon$ at each observation time.
The measurement grid improves the sampling of the maximum of the same Fourier interpolant.
The curves exhibit approximately quadratic dependence on $\varepsilon$.
Figure~\ref{fig:solution} displays the physical fields and their central sections from the same evolution at $T=6$ with $\varepsilon=0.1$.

\begin{figure}[htb!]
	\centering
	\includegraphics[width=0.8\linewidth]{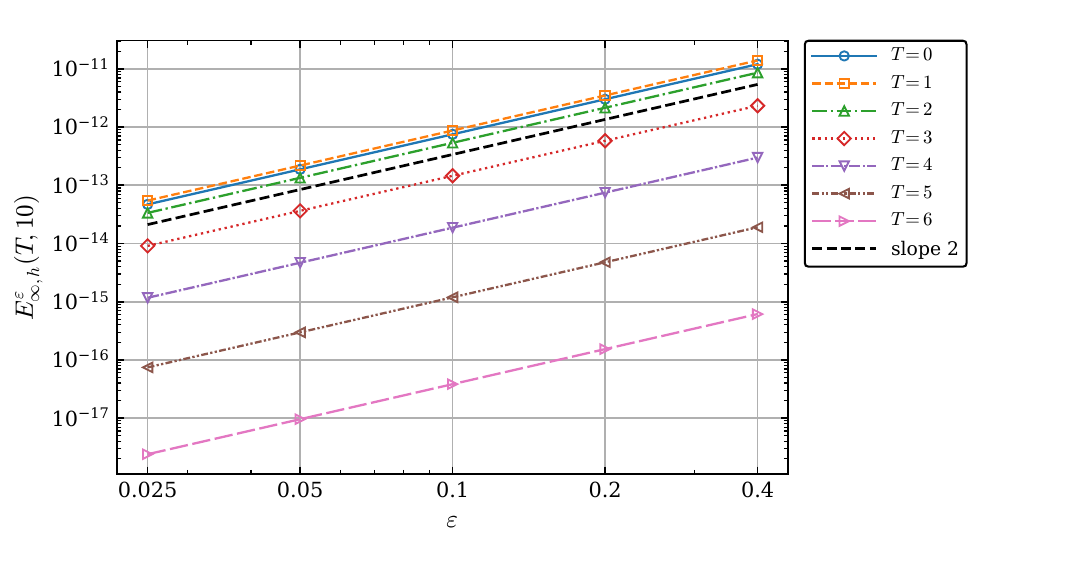}
	\caption{
	Sampled pointwise tail error with the baseline parameters in Section~\ref{sec:numerical-setup} and $T_{\texttt{ref}}=10$.
	The curves correspond to $T=0,1,2,3,4,5,6$, and the dashed black line has slope two.
	}
	\label{fig:order}
\end{figure}

\begin{figure}[htb!]
	\centering
	\includegraphics[width=\linewidth]{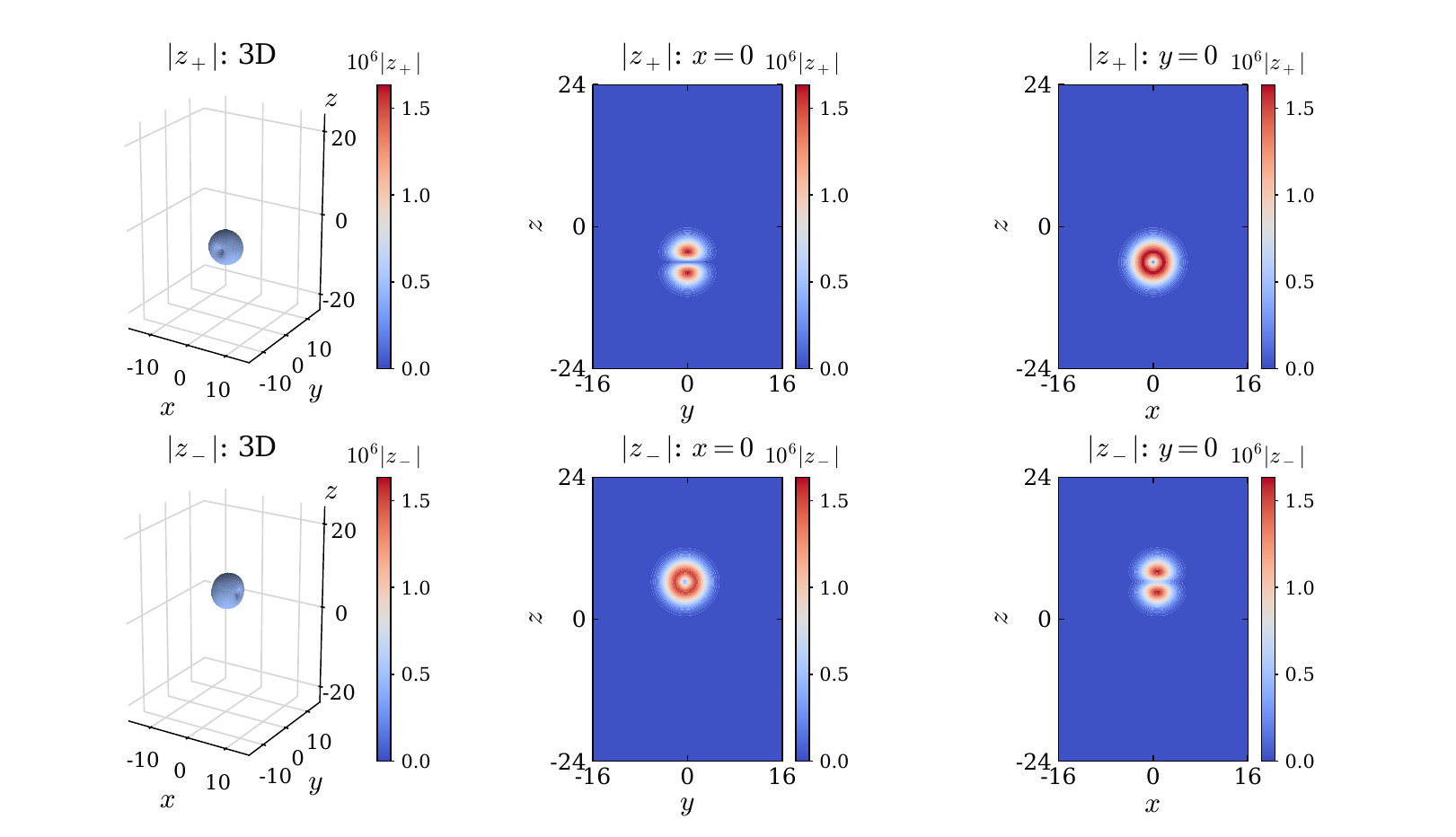}
	\caption{The numerical solution $|z_+|$ (top) and $|z_-|$ (bottom) at $T=6$ with $\varepsilon=0.1$:
	three-dimensional isosurfaces (left), $x=0$ sections (middle) and $y=0$ sections (right).}
	\label{fig:solution}
\end{figure}

For the sensitivity check, we vary one parameter at a time: halve the time step, refine the integration-label or physical grid, enlarge the domain, or increase the reference time.
The varied values are $\Delta t'=0.025$, $M_{\rm int}'=96\times96\times144$, $N'=(96,96,144)$, $L'=(48,48,72)$ and $T_{\texttt{ref}}'=12$.
Physical-grid refinement keeps both label grids fixed. Domain enlargement increases the physical and label grid sizes proportionally, preserving their respective spacings.
The time-step and grid/domain comparisons use $\varepsilon=0.4$ and $0.025$, all observation times and both reference times. The reference-time comparison uses all five amplitudes and all observation times.
For either measured norm $E$, we compute
\begin{equation*}
	\Delta_{\rm abs}=\max|E'-E|,\quad
	\Delta_{\rm rel}=\max\frac{|E'-E|}{E},\quad
	p_{\rm end}(T,T_{\texttt{ref}})=
	\frac{\log(E^{0.4}/E^{0.025})}{\log16},
\end{equation*}
Here $E=E^\varepsilon(T,T_{\texttt{ref}})$ denotes the baseline value and $E'$ with the prime denotes the varied setting. The maxima cover the comparison ranges specified above. We also compute $\Delta p_{\rm end}=p_{\rm end}'-p_{\rm end}$.
In both norms, the absolute and relative changes and the maximum endpoint-order changes are below $10^{-3}$. These small changes support the stability of the observed $\varepsilon^2$ scaling under the tested variations.

\begin{table}[htb!]
	\centering
	\caption{
	Sensitivity of the sampled pointwise tail error. The $p_{\rm end}$ column gives the endpoint amplitude order at $T=6$ with $T_{\texttt{ref}}=10$ or $T_{\texttt{ref}}=12$.
	The last column gives the maximum change over the compared times.
	}
	\label{tab:sanity}
	\begin{tabular}{lrrrr}
	\toprule
	 Setting & Max. absolute change & Max. relative change & $p_{\rm end}$ & $\max|\Delta p_{\rm end}|$\\
	 \midrule
	Baseline & --- & --- & 2 & --- \\
	$\Delta t/2$ & $6.74\times10^{-20}$ & $2.34\times10^{-7}$ & 2 & $2.18\times10^{-12}$ \\
	Integration labels & $2.58\times10^{-26}$ & $2.15\times10^{-15}$ & 2 & $8.88\times10^{-16}$ \\
	Physical mesh & $6.06\times10^{-18}$ & $5.18\times10^{-7}$ & 2 & $5.15\times10^{-10}$ \\
	Domain & $4.52\times10^{-16}$ & $3.33\times10^{-5}$ & 2 & $2.11\times10^{-8}$ \\
	$T_{\rm ref}$ & $6.35\times10^{-22}$ & $2.14\times10^{-7}$ & 2 & $7.27\times10^{-8}$ \\
	\bottomrule
	\end{tabular}
\end{table}

\subsection{Unweighted $H^6$ tail estimate}\label{sec:sobolev-numerics}
For the same vector tails, we measure the unweighted $H^6$ norm as follows
\begin{equation*}
 E_{6,h}^{\varepsilon}(T,T_{\texttt{ref}})
 :=\sum_\pm\left(\sum_{|\alpha|\le6}\int_\Omega
 |\partial_\xi^\alpha Q_{h,\pm,T}^{\varepsilon,T_{\texttt{ref}}}(\xi)|^2\,d\xi
 \right)^{1/2}.
\end{equation*}
Figure~\ref{fig:sobolev-order} shows the results with the same parameters as Figure~\ref{fig:order}.
The adjacent-amplitude orders, computed as $\log(E_{6,h}^{\varepsilon}/E_{6,h}^{\varepsilon/2})/\log 2$ at fixed $T$ and $T_{\texttt{ref}}$, are all close to two.

\begin{figure}[htb!]
	\centering
	\includegraphics[width=0.8\linewidth]{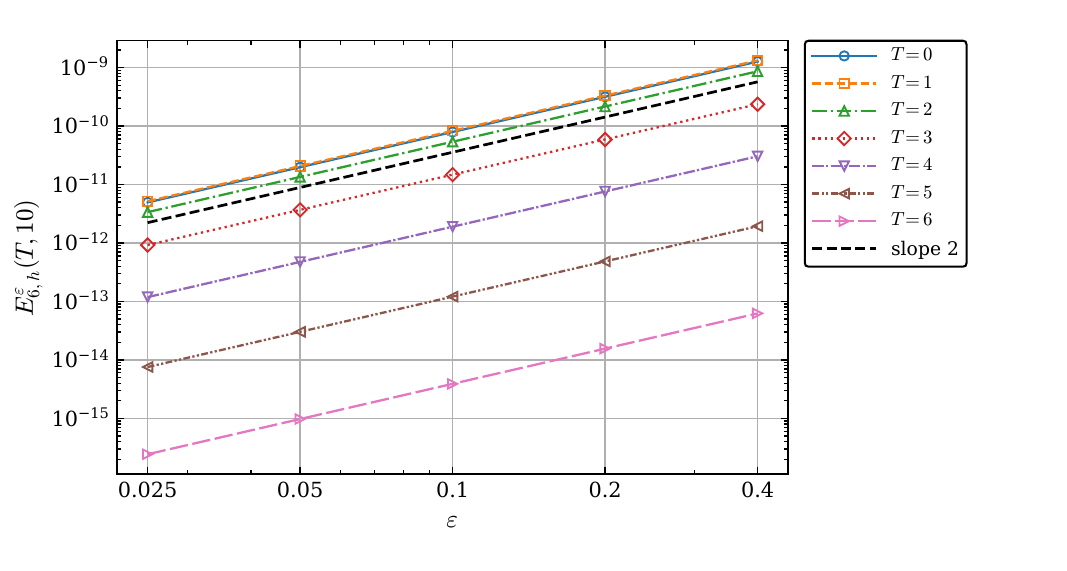}
	\caption{
	Unweighted $H^6$ tail error of the vector tails as in Figure~\ref{fig:order}, with the same evolution parameters, amplitudes and comparison times.
	}
	\label{fig:sobolev-order}
\end{figure}

Table~\ref{tab:sobolev-sensitivity} records the same sensitivity comparisons for $E_{6,h}^{\varepsilon}$ as Table~\ref{tab:sanity}. The small changes show that the measured $H^6$ tail errors and their amplitude scaling are stable under these parameter variations.

\begin{table}[!htb]
	\centering
	\caption{
	Sensitivity of the unweighted $H^6$ tail error.
	The parameter variations are the same as in Table~\ref{tab:sanity}.
	}
	\label{tab:sobolev-sensitivity}

	\begin{tabular}{lrrrr}
		\toprule
		 Setting & Max. absolute change & Max. relative change & $p_{\rm end}$ & $\max|\Delta p_{\rm end}|$\\
		 \midrule
		Baseline & --- & --- & 2 & --- \\
		$\Delta t/2$ & $9.94\times10^{-18}$ & $2.28\times10^{-7}$ & 2 & $1.31\times10^{-12}$ \\
		Integration labels & $4.14\times10^{-25}$ & $4.80\times10^{-16}$ & 2 & $4.44\times10^{-16}$ \\
		Physical mesh & $1.04\times10^{-16}$ & $7.85\times10^{-8}$ & 2 & $1.64\times10^{-10}$ \\
		Domain & $2.05\times10^{-16}$ & $1.55\times10^{-7}$ & 2 & $5.02\times10^{-10}$ \\
		$T_{\rm ref}$ & $4.12\times10^{-20}$ & $7.92\times10^{-8}$ & 2 & $2.69\times10^{-8}$ \\
		\bottomrule
	\end{tabular}
\end{table}

\section{Concluding remarks}

Building on the global weighted estimates in \cite{Li2023},
we have established quantitative finite-time approximations of the future scattering fields in both $L^\infty$ and unweighted $H^{N_*+1}$,
with error bounded by $C_\delta\varepsilon^2(R+T)^{-\delta}$.
The pointwise estimate follows by integrating the pressure decay,
while the Sobolev estimate also uses bounds for derivatives of the characteristic flows.

The numerical experiments measure the same pressure-integral tails in the sampled pointwise norm and the unweighted $H^6$ norm,
which support the $\varepsilon^2$ dependence.
The observed scaling remains consistent under the tested variations in the time step, spatial resolution, domain size, and reference time.
It would be interesting to investigate whether sharper decay assumptions,
different spatial weights, or additional structure of the initial perturbation can yield faster finite-time rates.

\bigskip

\noindent
\textbf{Ethics declarations}: 
The authors declare that there is no conflict of interest.

\bigskip

\noindent
\textbf{Acknowledgements}:
This work is supported in part by  Natural Science Foundation of Jiangsu Province, Jiangsu Provincial Scientific Research Center of Applied Mathematics (Grant No. BK20233002), National Natural Science Foundation of China (No. 12671481),  1+1+1 CUHK-CUHK(SZ)-GDSTC Joint Collaboration Fund
(No. 2025A0505000079), University Development Fund (No. UDF01004388), and Alexander von Humboldt Foundation Research Fellowship (No. CHN-1249639-HFST-P, No. CHN-1234352-HFST-P).


\begin{thebibliography}{99}
\small
\bibitem{Alfven1942}
H.~Alfv\'en,
Existence of electromagnetic-hydrodynamic waves,
\emph{Nature} \textbf{150} (1942), 405--406.
\url{https://doi.org/10.1038/150405d0}.

\bibitem{BardosSulemSulem1988}
C.~Bardos, C.~Sulem and P.-L. Sulem,
Longtime dynamics of a conductive fluid in the presence of a strong magnetic field,
\emph{Trans. Amer. Math. Soc.} \textbf{305} (1988), 175--191.
\url{https://doi.org/10.1090/S0002-9947-1988-0920153-5}.



 
\bibitem{HeXuYu2018}
L.-B. He, L. Xu and P. Yu,
On global dynamics of three dimensional magnetohydrodynamics: nonlinear stability of Alfv\'en waves,
\emph{Ann. PDE} \textbf{4} (2018), Art.~5, 105 pp.
\url{https://doi.org/10.1007/s40818-017-0041-9}.
 
\bibitem{Friedlander1980}
F.~G. Friedlander,
Radiation fields and hyperbolic scattering theory,
Math. Proc. Cambridge Philos. Soc. \textbf{88} (1980), 483--515.


\bibitem{Lawson1967}
J.~D. Lawson,
Generalized Runge--Kutta processes for stable systems with large Lipschitz constants,
\emph{SIAM J. Numer. Anal.} \textbf{4} (1967), 372--380.
\url{https://doi.org/10.1137/0704033}.

\bibitem{LaxPhillips1967}
P.~D. Lax and R.~S. Phillips,
\emph{Scattering Theory},
Pure and Applied Mathematics, Vol.~26, Academic Press, New York--London, 1967.

\bibitem{Leray1934}
J.~Leray,
Sur le mouvement d'un liquide visqueux emplissant l'espace,
\emph{Acta Math.} \textbf{63} (1934), 193--248.
\url{https://doi.org/10.1007/BF02547354}.

\bibitem{Li2023}
M. Li,
Inverse scattering of Alfv\'en waves in three dimensional ideal magnetohydrodynamics,
\emph{Adv. Math.} \textbf{435} (2023), Part A, Art.~109363, 69 pp.
\url{https://doi.org/10.1016/j.aim.2023.109363}.

 

\bibitem{LiYu2021}
M. Li and P. Yu,
On the rigidity from infinity for nonlinear Alfv\'en waves,
\emph{J. Differential Equations} \textbf{283} (2021), 163--215.
\url{https://doi.org/10.1016/j.jde.2021.02.036}.




	
	
	
%
	
	
\bibitem{MaronGoldreich2001}
J.~Maron and P.~Goldreich,
Simulations of incompressible magnetohydrodynamic turbulence,
\emph{Astrophys. J.} \textbf{554} (2001), 1175--1196.
\url{https://doi.org/10.1086/321413}.

\bibitem{NielsonHowesDorland2013}
K.~D. Nielson, G.~G. Howes and W.~Dorland,
Alfv\'en wave collisions, the fundamental building block of plasma turbulence.
II. Numerical solution,
\emph{Phys. Plasmas} \textbf{20} (2013), Art.~072303.
\url{https://doi.org/10.1063/1.4812807}.




 
	
	
	\bibitem{Orszag1971}
	S.~A. Orszag,
	Numerical simulation of incompressible flows within simple boundaries. I. Galerkin (spectral) representations,
	\emph{Stud. Appl. Math.} \textbf{50} (1971), 293--327.
	\url{https://doi.org/10.1002/sapm1971504293}.
	



 

\bibitem{ShestovVoitenkoZhukov2022}
S.~V. Shestov, Y.~M. Voitenko and A.~N. Zhukov,
Initiation of Alfv\'enic turbulence by Alfv\'en wave collisions: A numerical study,
\emph{Astron. Astrophys.} \textbf{661} (2022), Art.~A93.
\url{https://doi.org/10.1051/0004-6361/202142362}.

\bibitem{VernieroHowesKlein2018}
J.~L. Verniero, G.~G. Howes and K.~G. Klein,
Nonlinear energy transfer and current sheet development in localized
Alfv\'en wavepacket collisions in the strong turbulence limit,
\emph{J. Plasma Phys.} \textbf{84} (2018), Art.~905840103.
\url{https://doi.org/10.1017/S0022377817001003}.


%
 
\end{thebibliography}
\end{document}